\documentclass{article}

\usepackage{amssymb,latexsym,amsmath,amsthm,amsfonts,graphics}
\usepackage{graphicx}
\graphicspath{ {Figures/} }

\usepackage{caption}
\usepackage{subcaption}
\usepackage[rightcaption]{sidecap}

\usepackage{color}
\usepackage{lineno}
\usepackage{multirow}
\usepackage{epstopdf}
\usepackage{rotating}
\usepackage{cite}

\usepackage[a4paper, total={6.8in, 9in}]{geometry}
\usepackage{hyperref}
\usepackage{tikz}

\newtheorem{thm}{Theorem}[section]

\newtheorem{lem}{Lemma}[section]

\newtheorem{rem}{Remark}[section]
\newcommand{\N}{\mathbb{N}}

\makeatletter
\def\ps@pprintTitle{%
 \let\@oddhead\@empty
 \let\@evenhead\@empty
 \def\@oddfoot{\centerline{\thepage}}%
 \let\@evenfoot\@oddfoot}
\makeatother

\begin{document}


\begin{center}
{\bf {New bounds on the number of Edges of the maximal Permutation Graphs}}
\end{center}
\begin{center}
{\bf M. Anwar* \footnote {$``<{\rm mohamedanwar} \ @\ {\rm sci.asu.edu.eg}>''$}}
, Mahmoud Tarek*\footnote {$``<{\rm }_{-}{\rm } \ @\ {\rm sci.asu.edu.eg}>''$}
, Ahmed Gaber* \footnote{$``<{\rm a.gaber} \ @\ {\rm sci.asu.edu.eg}>$}\\
{\footnotesize *Department of Mathematics, Faculty of Science, Ain Shams University, 11566 Abbassia, Cairo, Egypt.}
\end{center}

\begin{abstract}
In this paper, we bound the number of edges of a maximal permutation graph with n vertices. We propose a new method to compute the lower bound by splitting the set of labellings of the edges into six parts, considering one separate problem for each part, explicitly determining the cardinality of four parts and summing up the corresponding values. We finish with an upper bound of the number of edges of a maximal permutation graph.
\end{abstract}



\section{Introduction}

A graph labeling, informally, is an assignment of integers to the vertices or edges, or both, subject to specified constraints. Graph labelings were first introduced in the mid 1960s. In the intervening years over 200 graph labelings techniques have been studied in over 3000 papers. Most graph labeling methods trace their origin to one introduced by Rosa [7] in 1967, or one given by Graham and Sloane [4] in 1980. The seminal survey by Gallian [3] is a basic reference devoted to collect a large number of known labelings of graphs.\\

The theory of labelings of graphs has a wide range of applications, for instance, x-ray, crystallography, coding theory, cryptography, data base management, radar , astronomy, circuit design, communication design and Convolution Codes with optimal autocorrelation properties. See [1, 2, 6, ?].\\

The permutation labeling of graphs was first studied by Hegde and Shetty in [5]. They defined a graph $G$ with $p$ vertices and $q$ edges to be a permutation graph if there exists a bijection function $f : V (G) \rightarrow \{1, 2, \dots, p\}$ such that the induced edge function $h_f : E(G) \rightarrow\N$ is defined as follows
\begin{equation*}
  h_f(x_ix_j)=\left\{
                \begin{array}{ll}
                  P^{f(x_i)}_{f(x_j)}, &  \ \hbox{if $f(x_i)>f(x_j)$;} \\
                  P^{f(x_j)}_{f(x_i)}, &  \ \hbox{if $f(x_j)>f(x_i)$.}
                \end{array}
              \right.
\end{equation*}

In other words, a simple graph $G$ with $n$ vertices is called a permutation graph if its vertices can be labeled with distinct integers $1, 2, ..., n$ such that when each edge $xy$, where $y > x$, is labeled with $P_x^y$ , the induced edge labels will be distinct. They proved that the complete graph $K_n$ is a permutation graph if and only if $n \leq 5$ . Many well known families of graphs are known to be permutation graphs, see Anwar, for example. A permutation graph $G$ is said to be a maximal permutation graph if adding any new edge breaks the permutability of $G$, see [8]. We note that there may be many non-isomorphic maximal permutation graphs of the same number of vertices. For example, consider the two graphs\\

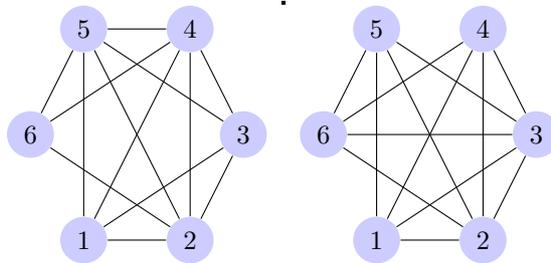
\begin{center}

\begin{frame}

\begin{tikzpicture}
  [scale=0.7,auto=center,every node/.style={circle,fill=blue!20}]
  \node (v1) at (0,0)    {$1$};
  \node (v2) at (2,0)    {$2$};
  \node (v3) at (3,2)    {$3$};
  \node (v4) at (2,4)    {$4$};
  \node (v5) at (0,4)    {$5$};
  \node (v6) at (-1,2)   {$6$};
  \draw (v1) -- (v2);
  \draw (v1) -- (v3);
  \draw (v1) -- (v4);
  \draw (v1) -- (v5);

  \draw (v2) -- (v3);
  \draw (v2) -- (v4);
  \draw (v2) -- (v5);
  \draw (v2) -- (v6);

  \draw (v3) -- (v4);
  \draw (v3) -- (v5);

  \draw (v4) -- (v5);
  \draw (v4) -- (v6);

  \draw (v5) -- (v6);
\end{tikzpicture}
\hspace{0.2cm}
\begin{tikzpicture}
  [scale=0.7,auto=center,every node/.style={circle,fill=blue!20}]
  \node (v1) at (0,0)    {$1$};
  \node (v2) at (2,0)    {$2$};
  \node (v3) at (3,2)    {$3$};
  \node (v4) at (2,4)    {$4$};
  \node (v5) at (0,4)    {$5$};
  \node (v6) at (-1,2)   {$6$};
  \draw (v1) -- (v2);
  \draw (v1) -- (v3);
  \draw (v1) -- (v4);
  \draw (v1) -- (v5);

  \draw (v2) -- (v3);
  \draw (v2) -- (v4);
  \draw (v2) -- (v5);
  \draw (v2) -- (v6);

  \draw (v3) -- (v4);
  \draw (v3) -- (v5);
  \draw (v3) -- (v6);

  \draw (v4) -- (v6);

  \draw (v5) -- (v6);
\end{tikzpicture}
\captionof{figure}{}\label{figure3}
\end{frame}
 \end{center}
Clearly, the two graphs are non-isomorphic maximal permutation graphs.\\

The aim of this paper is twofold. First, we investigate the least number of edges in a maximal permutation graph. To this end, we establish, in section 2, many lemmata that relates prime numbers and permutations. In section 3, we obtain our main result concerning a lower bound of the number of edges in a maximal permutation graph. Second, in section 4, we obtain an upper bound for the number of edges in a maximal permutation graph. We round off ...

\section{Prime numbers and permutations}

In this section, we prove some lemmas that paves the way to establish our main result.
\begin{lem}
  If $ m! < n $, then $P_{r}^{h} < P_{k-m}^{k} , P_{k-m+1}^{k}, \dots , P_{k-1}^{k}$  for all $ r \leq h < k$.
\end{lem}
\begin{proof}
  Since $m! < k$, it is clear that $m![(m + 1)(m + 2)\dots(k - 1)] < k[(m + 1)(m +2)\dots(n - 1)]$. Therefore, 
  $P_{r}^{h} \leq P_{k-2}^{k-1} < P_{k-m}^{k}$ for all $r\leq h< k$.
\end{proof}

\begin{lem}
  Let $q$ be a prime number. If $q > s + \sqrt[]{s + 1}, \ s\geq 2$ , then $P^{q+1}_s\neq P^r_k$ for all $r > s$.
\end{lem}
\begin{proof}
 Let $q$ be a prime number and $r > s$. Clearly, $P^{q+1}_s\neq P^r_k$ in the case of $k > q$. Now, Suppose by contrary that $P^{q+1}_s = P^r_k$ . If $n < q$, then $P^k_r \not\equiv 0(\mod q)$ and since $P^{q+1}_s \equiv 0(q)$ for all $s \geq 2$, therefore $q\mid P^{q+1}_s = P^k_r$ which is a contradiction.If $k = q$, then $P^{q+1}_s = P^q_r$ . Hence, $(q + 1)\dots (q - s+ 2) = (q)\dots(q - s + 1)(q - s)\dots (q - r + 1)$. Consequently, $q+1 = (q-s+1)...(q-r+1)$. In particular, $q+1\geq(q-s+1)(q-s)$.Hence, $q^2 - 2sq + (s^2 - s - 1) \leq 0$. Since, the zeros of $f(x)=x^2 - 2sx + (s^2 -s-1)$ are $s\pm \sqrt{s + 1}$ and $f(x)$ is increasing for all $x \geq s$, therefore $f(x)>0$ for all $x > s + \sqrt[]{s + 1}$ which is a contradiction and the proof is complete.
\end{proof}

\begin{lem}
   Let $q$ be a prime number such that $q > 4s$, $s \geq 3$. Then $P^{q+2}_s \neq P^k_r$ for all $r > s$.
 \end{lem}
\begin{proof}
Let $q$ be a prime number and $r > s \geq 3$. Clearly, $P^{q+2}_s < P^k_r$ in the case of $n > q + 1$ and $P^{q+2}_s \neq P^k_r$ in the case of $k < q$ because of $P^k_r (\mod q)$ and $Pq+2 s \equiv 0( \mod q)$. 
Now, Suppose by contrary that $P^{q+2}_s = P^k_r$. If $k = q + 1$, using the same argument that has been used in lemma 3.3, one can get a contradiction. If $k = q$, then $P^{q+1}_s = P^q_r$ which implies that $(q + 2)(q + 1)\geq(q - s + 2)(q - s + 1)(q - s)$, that is $q^3 - (3s-2) q^2 - (s^2+1) q +(s^3 -3s^2 +2s-2) \leq 0$. Now, consider $f(x)=x^3 - (3s-2)x^2 - (s^2+1)x + (s^3 - 3s^2 + 2s - 2) \leq 0$. Easily, one can prove that $\acute{f}(x) > 0$ for all $x \geq 4s$ and $f(4s) = 13s^3 + 29s^2 - 2s - 2 > 0$ for all $s \geq 3$. Therefore, $f(x) > 0$ for all $x \geq 4s$. Which is a contradiction.
\end{proof}

\begin{lem}
  For any prime $q$ and for all $l$, $k \in\N$, we have $P^{q^mk+ql-1}_{ql-1}, \ P^{q^m(k+l)-1}_{ql-1} \neq P^h_r$ for all $r \geq ql$, where $m = l + v_q(l!)$.
\end{lem}
\begin{proof}
  It’s clear that for any $r\geq ql$, we have $v_q(P^h_r)\geq v_q(ql!)=l + \nu_q(l!)= m$ and $v_q(P^{q^mk + ql - 1}_{ql-1})= v_q(qmk + ql - 1) + v_q(q^mk + ql - 2) +\dots + v_q(q^mk + 1)$. Since, $v_q(x + y) = \min(v_q(x), v_q(y))$ in the case of $v_q(x)\neq v_q(y)$ and $v_q(q^mk)\geq m$ for all $j \leq qn - 1$, therefore $v_q(P^{q^mk + ql - 1}_{ql - 1}) = v_q(ql - 1) + v_q(ql - 2) +\dots + v_q(1) = v_q((ql - 1)(ql - 2)\dots (1)) = v_q((ql - 1)!) < v_q(ql!) = m$. Hence, $P^{q^mk + ql - 1}_{ql - 1}\neq P^h_r$ for all $r \geq ql$. The second part $P^{q^m(k+l) - 1}_{ql - 1}\neq P^h_r$ can be done in the same manner.
\end{proof}

\begin{lem}
  For any prime power $q^h \neq 3^1$, we have $P^{2q^h}_1 \neq P^k_r$ for all $r > 1$, where $h$ is a natural number.
\end{lem}
\begin{proof}
  Let $q^h \neq 3^1$ and $r > 1$. Suppose by contrary that, $P^{2q^h}_1=P^k_r$ . If $q = 2$, then there exist $a, b$ with $a > b$ such that $2a$ and $2b$ are consecutive integers, which is impossible unless $b = 0$. Therefore, the only solution is $P^2_1 \neq P^k_r$ for all $r > 1$, which is a contradiction. If $q\neq2$, then there exist $a, b$ such that $2q^a$ and $q^b$ are consecutive integers greater than $1$. Consequently, $2q^a - q^b = \pm 1$. Let $c = \min(a, b)$, therefore $qc(2q^{a-c} - q^{b-c}) = \pm 1$ which implies that $q^c = 1$ and hence $c = 0$. Therefore, we have $2q^a - 1 = \pm 1$ or $2 - q^b = \pm 1$. So, the only solution will exist if $2 - q^b = -1$, that is, when $q = 3$ and $b = 1$. Which is a contradiction and the proof is completed.
\end{proof}

\section{A lower bound for $|E(G(N))|$}
 In this section we introduce a lower bound of the number of edges of a maximal permutation graph with $n$ vertices. Define $S_1(n) := \{P^k_i : i = k - m_k, ..., k - 1, k = 3, ..., n\}$, where $m_k$ is defined to be the integer satisfies $m_k! < k \leq (m_k + 1)!$.
$S_2(n):= \{P^{q+1}_s : q \ be \ a \ prime \ number, \ s + \sqrt(s + 1) < q < n \ and \ s\geq3\}$.\\
$S_3(n):= \{P^{q+2}_s : q \ be\  a \  prime \ number, \ 4s < q < n - 1 \ and \ s\geq3\}$. \\
$S_4(n):= \{P^{q^mk+ql-1}_{ql-1} : q \ be \ a \ prime \ number \ and \ l, k \in\N \ such \ that \ q^mk + ql - 1 < n, \ ql < n, \ where \ m := l + v_q(l!)\}$.\\
$S_5(n):= \{P^{q^m(k+l)-1}_{ql-1} : q \ be \ a \ prime \ number \ and \ l, k\in\N \ such \ that \ q^m(k + l) - 1 < n, \ ql < n, \ where \ m := l + v_q(l!)\}$.\\
$S_6(n):= \{P^{2q^h}_1 : q \ be \ a \ prime \ number, \ q^h\neq 3^1 \ and \ h\in\N\} \cup P^2_1$.\\
\begin{thm}
  The number of edges of a maximal permutation graph with $n$ vertices is greater than 
  $ nm - m! - 4 - \sum\limits_{k=3}^{m-1}  k!
  + \sum\limits_{s=2}^{\lfloor\frac{2n+3-\sqrt{4n+18}}{2}\rfloor} \pi(n - 1) - \pi(s +\sqrt{s + 1}) +|S_4(n)\cup S_5(n)| +\sum\limits_{s=2}^{\lfloor\frac{n-1}{4}\rfloor}\pi(n - 1) - \pi(4s) + \sum\limits_{q \ is \ prime}^{q\leq\lfloor\frac{n}{2}\rfloor} \lfloor\log_q(\frac{n}{2})\rfloor - \delta(S_1(n), \dots, S_6(n))$, where $\delta(S_1(n), \dots, S_6(n))$ means the sum of (repetitions-1) of each set label.
\end{thm}
\begin{proof}
 Using the above notation and the lemmas that have been introduced in the previous section, one can see that the number of edges in the maximal permutation graph G with n vertices is greater that $|\cup^{m}_{n =1} S_i(n)| - \delta(S1(n), \dots, S6(n))$. Let us try to count the number of elements in $S_1$, $S_2$, $S_3$, $S_6$. Let us partition the numbers in $S_1$ with respect to

\begin{eqnarray*}
 |S_1(n)| &=& \sum\limits_{k=3}^{n} \sum\limits_{i=k-m_k}^{k-1} 1\\
          &=& \sum\limits_{3}^{n}  m_k\\
          &=& \sum\limits_{k=3}^{3\!}2+\sum\limits_{k=3\!+1}^{4\!}3+ \dots+\sum\limits_{k=(m_n-1)\!}^{4\!}m_n-1+\sum\limits_{m_n\!+1}^{n}m_n \\
          &=& 2(3\! - 2\!) + 3(4\! - 3\!)+  \dots+ (mn - 1)(mn\! - (mn - 1)\!) + mn(n - mn\!) \\
          &=& nm - m\! - 4 -\sum\limits_{k=3}^{m-1}
\end{eqnarray*}
For a fixed $s$ the number of elements will be $\pi(n - 1) - \pi(s +\sqrt{s + 1})$ and the number of possible $s$ is when the conditions of the prims fill. In that case $n = s - 2 + \sqrt{s + 1}$. Consequently, $(n+2)^2-2s(n+2)+s^2 = s+1$. Therefore, $s^2-(2n+5)s+(n+2)^2-1$. Consequently, $s =\frac{2n+5\pm\sqrt{4n^2+20n+25-4(n^2+4n+2)+1}}{2} = \frac{2n+3\pm\sqrt{4n+18}}{2}$. Therefore, $s=\frac{2n+3-\sqrt{4n+18}}{2}$. Therefore, $|s_2(n)| =\sum\limits_{s=2}^{\lfloor \frac{2n+3-\sqrt{4n+18}}{2}\rfloor}\pi(n - 1) - \pi(s +\sqrt{s + 1})$.\\

Now, consider the number of elements constructed by lemma 3.3 $(s_3)$. For a fixed $s$ the number of elements will be $\pi(n - 1) - \pi(4s)$ and the number of available values of $s$ is from $s = 3$ until, when the conditions of the prims fill. In that case $n - 2 = 4s + 1$
therefore, $s = \lfloor\frac{n-3}{4}\rfloor$. Therefore, $|s3(n)| = \sum\limits_{s=2}^{\lfloor\frac{n-3}{4}\rfloor}\pi(n - 1) - \pi(4s)$.\\

Now, consider the number of elements constructed by lemma 3.5 $(s_6)$. For fixed $q$. If $2q^h\leq n$ then, $h\leq\lfloor\log_q \frac{n}{2}\rfloor$. Therefore, the number of elements constructed by $q$ is $\lfloor\log_q \frac{n}{2}\rfloor$  if $q \neq 3$, $\lfloor\log_q \frac{n}{2}\rfloor-1$ if $q = 3$ and $\lfloor\log_q \frac{n}{2}\rfloor + 1$ if $q = 2$ considered $h$ can be zero if $q = 2$ and can not be $1$ if $q = 3$. So, in the total we can say that $|s_6(n)| = \sum\limits_{q \ is \ prime}^{q\leq\lfloor\frac{n}{2}\rfloor}\lfloor\log_q \frac{n}{2}\rfloor$.\\

In total without considering that the elements constructed by a lemma can be constructed by another lemma we have that. The maximal number of edges of a permutation graph with n vertices is greater than
$|E(G(n))| > nm - m! - 4 - \sum\limits_{k=3}^{m-1}  k!
 +\sum\limits_{s=2}^{\lfloor \frac{2n+3-\sqrt{4n+18}}{2}\rfloor} \pi(n - 1) - \pi(s +\sqrt{s + 1})+ |S_4(n)\cup S_5(n)| +\sum\limits_{s=2}^{\lfloor\frac{n-1}{4}\rfloor} \pi(n-1)-\pi(4s)+ \sum\limits_{q \ is \ prime}^{q\leq\frac{n}{2}}\lfloor\log_q\frac{n}{2} \rfloor-\delta(S1(n), ..., S6(n))$.
\end{proof}

\begin{rem}
  For large values of $n$ the lower bound can be modified by using good approximations of the
function $\pi(n)$.
\end{rem}
\begin{thm}
   $S_{1} \cap S_{i} = \phi$  for all $i = 2, 3, 4, 5, 6$.
\end{thm}
\begin{proof}
  We will prove that fact by showing that there are no intersections between $S_{1}(n)$
and other sets. Its clear that there is no intersection with $S_{6}$. For $S_{2}$, we have $s = \lfloor \frac{2n+3-\sqrt{4n+18}}{2}   \rfloor$.
Therefore, $s\leq n-\sqrt{n}-(\sqrt{n}-2)$. In fact, $m!> m^{2}$ for all $m\geq 4$. Thus, we have $\sqrt{n}> m$. Also, $\sqrt{n}> 2$ for all $n> 4.$ Consequently, $s< n-m-(\sqrt{n}-2)< n-m.$ Note also the first possible value for $n$ is at $n = 6$ in $S_{2}$. So, there is no intersection with $S_{2}$.
The sanme proof also works for $s_{3}$. Now consider $S_{4}$.  We can say that there is no
intersection with $S_{1}$. In fact $q^{h} > ql - 1$ for all $ql > 2$ if $ql - 1 \geq \frac{n}{2}$
 then $q^{h} + ql - 1 > n$.
. Also, $2q^{h} - 1 > n$.
\end{proof}

 \section{A lower bound for $|E(G(N))|$}
 In this section we introduce a new upper bound of the number of edges of a maximal
permutation graph. Let $G(n)$ be a maximal permutation graph with $n$ vertices.
Define $W_{h}(n) := \{k : P^{k}_1=P^{i}_h, k, i = 3, . . . , n\}.$
\begin{thm}
  $|W_{h}(n)| = i_{h} - h$, where $i_{h}$ is the largest integer with $i_{h} > h$ and
$n - i_{h}(i_{h} - 1)...(i_{h} - h + 1)\geq 0$.
\end{thm}
\begin{proof}
 We need to determine the value of $i: h< i < n$ and $i(i - 1)...( - h + 1) \leq n$.
 and $(i+1)i...(i-h+2) > n$. From the first inequality we can look for the largest integer that
makes $f(x) = x(x - 1)...(x - h + 1) \leq n$ or the largest integer $x$ with $h < x < n$
and $0 \leq n - x(x - 1)...(x - h + 1) < n $. Let $i_{h}$ be a solution to the system. Then
$W_{h}(n) = \{i_{h}, i_{h} -1, ..., h + 1\}$. Therefore, $|W_{h}(n)| = i_{h} - h$ .
\end{proof}
\begin{rem}
 The value of $i_{h}$ can be obtained by plotting the function $g(x) := n-x(x-1)...(x - h + 1) $ and finding the value of $i_{h}$ with the above required conditions.
\end{rem}

  \begin{thm}
    $|E(G(N))|\leq \frac{n(n-1)}{2}- \displaystyle \sum_{h=2}^{m_{n}-1}(i_{h} - h)$
  \end{thm}
\begin{proof}
  For any fixed $h$ and any element $k \in W_{h}(n)$ there is a unique $i$ such that $P^{k}_1=P^{i}_h$. Therefore, one the two edges $(k1)$ and $(ih)$ must be deleted and the result follows form
Theorem 4.1.
\end{proof}



\end{document}